\documentclass[12pt]{amsart}
\usepackage[colorlinks=true,pagebackref,hyperindex,citecolor=Green,linkcolor=Blue]{hyperref}
\usepackage[top=1in, bottom=1in, left=1.1in, right=1.1in]{geometry}
\usepackage{amsmath}
\usepackage{amsfonts}
\usepackage{amssymb}
\usepackage{amsthm}
\usepackage[all]{xy}
\usepackage{mathrsfs}
\usepackage{enumitem} 
\usepackage[T1]{fontenc}
\usepackage{color}
\usepackage{stmaryrd}
\usepackage{bigints}
\usepackage{comment}
\usepackage{bm}
\usepackage[english]{babel}

\usepackage{graphicx} 
\usepackage{tikz-cd} 
\usepackage{tikz} 

\makeatletter
\def\namedlabel#1#2{\begingroup
#2%
\def\@currentlabel{#2}%
\phantomsection\label{#1}\endgroup
}
\makeatother

\newtheorem{theorem}{Theorem}[section]

\newtheorem{corollary}[theorem]{Corollary}

\newtheorem{lemma}[theorem]{Lemma}

\newtheorem{proposition}[theorem]{Proposition}

\newtheorem{algorithm}[theorem]{Algorithm}

\theoremstyle{definition} 
\newtheorem{definition}[theorem]{Definition}

\newtheorem{example}[theorem]{Example}
\newtheorem{remark}[theorem]{Remark}
\numberwithin{equation}{subsection}

\definecolor{blue-violet}{rgb}{0.54, 0.17, 0.89}
\definecolor{Blue}{rgb}{0.01, 0.28, 1.0}
\definecolor{gGreen}{rgb}{0.2, 0.8, 0.2}
\definecolor{Green}{rgb}{0.04, 0.85, 0.32}

\makeatletter
\def\@tocline#1#2#3#4#5#6#7{\relax
  \ifnum #1>\c@tocdepth 
  \else
    \par \addpenalty\@secpenalty\addvspace{#2}%
    \begingroup \hyphenpenalty\@M
    \@ifempty{#4}{%
      \@tempdima\csname r@tocindent\number#1\endcsname\relax
    }{%
      \@tempdima#4\relax
    }%
    \parindent\z@ \leftskip#3\relax \advance\leftskip\@tempdima\relax
    \rightskip\@pnumwidth plus4em \parfillskip-\@pnumwidth
    #5\leavevmode\hskip-\@tempdima
      \ifcase #1
       \or\or \hskip 1.9em \or \hskip 2em \else \hskip 3em \fi%
      #6\nobreak\relax
    \dotfill\hbox to\@pnumwidth{\@tocpagenum{#7}}\par
    \nobreak
    \endgroup
  \fi}
\makeatother

\newcommand{\NN}{\mathbb{N}}
\newcommand{\RR}{\mathbb{R}}

\newcommand{\CC}{\mathbb{C}}

\newcommand{\cO}{{\mathcal O}}

\newcommand{\cJ}{{\mathcal J}}

\begin{document}

\title[Multiplier ideals of meromorphic functions]{Multiplier ideals of meromorphic functions in dimension two}

\author[M. Alberich-Carrami\~nana]{Maria Alberich-Carrami\~nana}
\address{Departament de Matem\`atiques  and  Institut de Matem\`atiques de la UPC-BarcelonaTech (IMTech)\\  Universitat Polit\`ecnica de Catalunya \\ Av.~Diagonal 647, Barcelona
08028; and Institut de Rob\`otica i Inform\`atica Industrial\\ CSIC-UPC \\ Llorens i Artigues 4-6, Barcelona 08028, Spain.} 
\email{Maria.Alberich@upc.edu}

\author[J. \`Alvarez Montaner]{Josep \`Alvarez Montaner}
\address{Departament de Matem\`atiques  and  Institut de Matem\`atiques de la UPC-BarcelonaTech (IMTech)\\  Universitat Polit\`ecnica de Catalunya \\ Av.~Diagonal 647, Barcelona
08028; and Centre de Recerca Matem\`atica (CRM), Bellaterra, Barcelona 08193.} 
\email{Josep.Alvarez@upc.edu}

\author[R. G\'omez]{Roger G\'omez-L\'opez}
\address{Departament de Matem\`atiques \\  Universitat Polit\`ecnica de Catalunya \\ Av.~Diagonal 647, Barcelona
08028}

\email{Roger.Gomez.Lopez@upc.edu}

\thanks{All three authors are partially supported by 
grant PID2019-103849GB-I00 funded by MICIU/AEI/
10.13039/501100011033
and AGAUR grant 2021 SGR 00603. JAM is also supported by Spanish State Research Agency, through the Severo Ochoa and Mar\'ia de Maeztu Program for Centers and Units of Excellence in R$\&$D (project CEX2020-001084-M). RGL gratefully acknowledges Secretaria d'Universitats i Recerca del Departament d’Empresa i Coneixement de la Generalitat de Catalunya and Fons Social Europeu Plus for the financial support of his FI Joan Oró predoctoral grant.}

\keywords {Multiplier ideals, jumping numbers, meromorphic functions.}

\subjclass[2000]{Primary 13D45, 13N10}

\begin{abstract}
We provide an effective method to compute multiplier ideals of meromorphic functions in dimension two. We also prove that meromorphic functions only have integer jumping numbers after reaching some threshold.

\end{abstract}

\maketitle

\section{Introduction}

Let $X$ be a $n$-dimensional complex  smooth algebraic variety and
$\cO_{X,O}$ the local ring of a point $O\in X$. Let  $f,g:(X,0) \longrightarrow (\CC,0)$ be two germs of holomorphic functions and consider the  germ of meromorphic function  $f/g: (X,0) \longrightarrow (\CC,0).$ Taking local coordinates we assume $f,g \in \cO_{X,O}= \CC\{ x_1,\dots , x_n\}$. Notice also that $f/g$ and $f'/g'$ define the same germ if $fg'=f'g$.

\vskip 2mm

In order to study the singularities of the meromorphic germ $f/g$, the authors in \cite{AGLN} introduced a theory of multiplier ideals mimicking the classical one for holomorphic functions (see also \cite{Take}). However this variant for meromorphic functions has several features that do not behave quite as in the classical way.  

\vskip 2mm

The aim of this note is to provide an effective method to compute multiplier ideals of meromorphic functions in dimension two. In Section \ref{Sec3} we present an algorithm to compute sequentially the set of jumping numbers and the corresponding multiplier ideals. This algorithm is a mild modification of the one given by the first two authors together with Dachs-Cadefau in \cite{AAD1} combined with the methods developed in \cite{AAB1, AAB2}. For completeness we will describe the differences between both cases.

\vskip 2mm

In Section \ref{Sec4} we pay attention to the version of Skoda's theorem for meromorphic functions considered in \cite{AGLN} which provides some interesting phenomena. For instance, we prove in Theorem \ref{thm:main} that after reaching some threshold the only jumping numbers of $f/g$ are integer numbers.

\vskip 2mm

In Section \ref{Sec5} we present some examples that illustrate the effectiveness of our methods and the properties proved in Theorem \ref{thm:main}. In particular, the examples show how the jumping numbers of $f/g$ depend on the contact between the two plane curves $f$ and $g$.

\section{Multiplier ideals of meromorphic functions} \label{Sec2}
Let $X$ be a $n$-dimensional complex  smooth algebraic variety and let  $U$ be a neighbourhood of $O\in X$. Let $\pi: Y \longrightarrow U$ be a log resolution of a meromorphic germ   $f/g: (X,0) \longrightarrow (\CC,0).$ Namely, $\pi$ is a proper birational map such that:

\begin{itemize}
\item $\pi$  is a log resolution of the hypersurface $H = \{f_{|_U} = 0\} \cup \{g_{|_U} = 0\}$, i.e., $\pi$ is an isomorphism outside a proper analytic subspace in $U$;

\item there is a 
normal crossing divisor $F$ on $Y$ such that $\pi^{-1}(H) = \mathcal{O}_Y(-F)$;

\item the lifting $\tilde{f} / \tilde{g} = (f/g)  \circ \pi = \frac{f \circ \pi}{g \circ \pi}$ defines a holomorphic map $\tilde{f} / \tilde{g} : Y \to \mathbb{P}^1$.

\end{itemize}

Let  $\{E_i\}_{i \in I}$ be the irreducible components of $F$. Some of the ingredients that we will use in the sequel are the {\it relative canonical divisor}, defined by the Jacobian  determinant of $\pi$, that we denote
$$K_\pi = \sum k_i E_i ,$$
and  {\it the total transforms} of $f$ and $g$  that we denote by
\[ \tilde{f}=\pi^* f= \sum N_{f, i} E_i, \quad \quad \tilde{g}= \pi^* g= \sum N_{g, i} E_i. \]
We point out that these divisors can be decomposed into their exceptional and affine parts, depending on whether the divisor $E_i$ has exceptional support or it corresponds to the components of the {\it strict transform}, which is the closure of $\pi^{-1}(C-\{O\})$ with $C$ being the hypersurface defined by $f=0$ or $g=0$ respectively.
We call $v_i(f):= N_{f, i}$ the \textit{value} of $f$ at $E_i$.

\vskip 2mm
Associated to the meromorphic germ $f/g$ we consider the  divisor $$\tilde{F} = \sum_{i \in I} N_{i} E_i, $$ with $  N_{i}:=N_{f,i}-N_{g,i}.$ 
We define
$$\tilde{F}_0=\sum_{N_{i} >0} N_{i}E_i,\quad \tilde{F}_\infty=\sum_{N_{i}<0} N_{i}E_i.$$

\begin{definition}\label{DefMultiplier}
Let $\pi: Y \longrightarrow U$ be a log resolution of a non-constant meromorphic germ   $f/g: (X,0) \longrightarrow (\CC,0).$
We define the $0$-multiplier ideal  of $f/g$ at $\lambda$ as the stalk at the origin of
\[\cJ \left( \left(\frac{f}{g}\right)^\lambda\right) =\pi_* \mathcal{O}_Y ( \lceil K_\pi - \lambda \cdot \tilde{F}_0 \rceil), \]
where $\lceil \cdot \rceil$ denotes the upper integer part of a real number or $\mathbb{R}$-divisor. If no confusion arises we denote the stalk at the origin in the same way and thus $\cJ((\frac{f}{g})^\lambda) \subseteq \cO_{X,O}.$
\end{definition}

A useful remark mentioned in \cite{AGLN} shows that we can describe the meromorphic multiplier ideal from the theory of mixed multiplier ideals $\cJ(f^{\lambda_1} g^{\lambda_2})$ associated to the pair of functions. Namely,
pick $t\in\NN$ such that $t\geq \lambda$.
Then,
$$
\cJ \left(\left( \frac{f}{g}\right)^\lambda\right)= \cJ\left(f^{\lambda} g^{t-\lambda}\right): g^{t} = \left\{ h \in \cO_{X,O} \;|\; hg^t \in \cJ\left(f^{\lambda} g^{t-\lambda}\right) \right\} .
$$

\begin{remark}
 Analogously, we define the $\infty$-multiplier ideal  of $f/g$ at $\lambda$ as the stalk at the origin of
\[\cJ^\infty \left( \left(\frac{f}{g}\right)^\lambda\right) =\pi_* \mathcal{O}_Y ( \lceil K_\pi + \lambda \cdot \tilde{F}_\infty \rceil) = \cJ \left( \left(\frac{g}{f}\right)^\lambda\right) . \] 
\end{remark}

Many of the properties of the classical multiplier ideals are still satisfied in the meromorphic case. For instance, the definition is independent of the chosen log resolution $\pi$ and there exists a discrete strictly increasing sequence of rational numbers
$\lambda_1 < \lambda_2 < \cdots$
such that 
\[\cJ\left(\left(\frac{f}{g}\right)^{\lambda_{i+1}}\right) \subsetneq \cJ\left(\left(\frac{f}{g}\right)^c\right) = \cJ\left(\left(\frac{f}{g}\right)^{\lambda_i}\right)\]
for $c \in [\lambda_i, \lambda_{i+1})$, and all $i$. These rational numbers are called the {\it $0$-jumping numbers} of the meromorphic function $f/g$. If no confusion may arise, we shall refer to them and to their corresponding $0$-multiplier ideals simply as jumping numbers and multiplier ideals.

\section{Explicit computation in dimension two} \label{Sec3}

The classical case of multiplier ideals of holomorphic functions has been extensively studied when the complex  smooth algebraic variety $X$ has dimension two. For simple complete ideals or irreducible plane curves, J\"arviletho \cite{Jar} and Naie \cite{Naie} provided a closed formula for the set of jumping numbers. For the case of any ideal or any plane curve we must refer to the work of Tucker \cite{Tuc} or the papers by the first two authors with Dachs-Cadefau and Alonso-Gonz\'alez \cite{AAD1, AADG1}. Hyry and J\"arviletho  \cite{HJ1,HJ2}  also worked on this general setting in dimension two.

\vskip 2mm

The method presented in \cite{AAD1} is an algorithm that computes sequentially at each step a jumping number and its associated multiplier ideal and thus it provides  the ordered sequence of multiplier ideals in any desired range of the real line. Combined with the effective methods developed by the first two authors with Blanco \cite{AAB1, AAB2} one may start the algorithm with the equation of a plane curve or generators of any ideal and get a set of  generators of each multiplier ideal as an output (see also \cite{BD}). In this section we will adjust  these methods to the case of meromorphic functions. We present first the algorithm and then we will explain each step with some detail highlighting the main differences with the classical case.

\begin{algorithm}[Jumping Numbers and Multiplier Ideals]\label{Alg: JN meromorphic}\hspace{45mm}

\vskip 1mm

\noindent {\tt Input:} A pair of holomorphic functions $f,g \in \mathcal{O}_{X,O}$.

\vskip 1mm
\noindent {\tt Output:} List of Jumping Numbers $\lambda_1 < \lambda_2 < \dots $ of $f/g$ and a set of generators of their corresponding multiplier ideals.

\vskip 2mm

\begin{itemize}[leftmargin=1.75cm]

\item[\textbf{(Step 1)}] Compute the minimal log resolution of $f/g$ and set $\tilde{F}_0= \sum_{N_{i} >0} N_{i}E_i$. 

 \vskip 2mm

 \item[\textbf{(Step 2)}]
Set $\lambda_0=0$, $e_i^{\lambda_0}=0$. From $j=1$, incrementing by $1$

\vskip 2mm
\begin{enumerate}
\setlength{\leftskip}{3cm}

\item[\textbf{(Step 2.$j$)}]

 \begin{enumerate}
  \item[$\cdot$] {\bf Jumping number}: Compute
    \vskip -2mm
    \[
      \lambda_j=\min_{N_i>0}\left\{\frac{k_i+1+e_i^{\lambda_{j-1}}}{N_i}\right\}.
    \]
    \vskip -1mm

  \item [$\cdot$] {\bf Multiplier ideal}: Compute the antinef closure $D_{\lambda_j}=\sum e_i^{\lambda_j}E_i$ of $\lfloor \lambda_j  \tilde{F}_0 - K_{\pi} \rfloor$ using the unloading procedure. $\cJ((\frac{f}{g})^{\lambda_j}) = \pi_*\mathcal{O}_Y(-D_{\lambda_j})$.

    \vskip 1mm

  \item [$\cdot$] {\bf Generators}: Give a system of generators of the complete ideal associated to $D_{\lambda_j}$. 
 \end{enumerate}

 \end{enumerate}

\end{itemize}
\end{algorithm}

\subsection{Computing the minimal log resolution of meromorphic functions}\label{ssec:log resolution}

Computing the log resolution of reduced plane curves is an already well-known procedure see, for instance \cite{CAbook}.
To obtain a log resolution of $f/g$, one can begin with a log resolution of the curve defined by $f \cdot g$ and then perform additional blow-ups to ensure a log resolution of the generic fibers of $f/g$, which is achieved by blowing up the base points of the ideal $(f,g)$.
These latter blow-ups also guarantee that each \textit{dicritical} component $E$ of $f/g$ has appeared, i.e. an exceptional divisor $E$ for which the restricted map $(\tilde{f} / \tilde{g})_{\vert_E} : Y \to \mathbb{P}^1$ is surjective (non-constant). 
Notice that if $E_i$ is a dicritical component then $N_{i}=0$.

In \cite{Alb}, an algorithmic procedure was described to compute the base points of $(f,g)$; see also \cite{AAB1}. To compute the log resolution of $f/g$, we may adapt Algorithm 4.6 from \cite{AAB1}, which is more convenient since it begins with the log resolution of $f \cdot g$, by omitting the last part (steps v to viii). The finiteness and correctness of this procedure are proved there.

\vskip 1mm

\begin{algorithm}[Minimal log resolution of $f/g$]\label{alg-log_resol}\hspace{45mm}

\vskip 1mm

\noindent {\tt Input:} A pair of holomorphic functions $f, g \in \mathbb{C}\{x,y\}$.

\vskip 1mm
\noindent {\tt Output:} Divisor $F= \sum_{i \in I} E_i$ of the minimal log resolution of $f/g$ described by means of the proximity matrix, and the values $(N_{f,i}, N_{g,i})_{i \in I}$ of $f$ and $g$.

\vskip 2mm

\begin{enumerate}
 \item Find $b=\gcd(f,g)$ and set $f=ba_1$, $g=ba_2$. Compute $h=ba_1a_2$.
 \item Find the exceptional divisor $F'= \sum_{i \in I'} E_i$ of the minimal log resolution of the reduced germ $\xi_{\textrm{red}}$, where $\xi: h=0$, and the values $v_{i} (a_k)$ and $v_{i} (b)$ at each $E_i$.
\\ 
 \noindent Compute $v_i=\min_i \{ v_{i} (a_1), v_{i} (a_2) \}$ for $i \in I'$.

 \item Define $F''= \sum_{i \in I''} E_i \geq F'$ by adding, if necessary, exceptional components of blow-ups of free points using the following criterion 
 and set, for each new $E_i$, $v_i=\min \{ v_{i} (a_1), v_{i} (a_2) \}$.

 {\textit{\textbf{Repeat}}} for $k \in \{1,2\}$:

\begin{enumerate}
 \item [$\cdot$] Blow-up the free point $q \in E_i$ on the strict transform of $a_k$,
\end{enumerate}

{\textit{\textbf{while}}} $v_i(a_k)=v_i$.
 
 \item  Define $F= \sum_{i \in I} E_i \geq F''$ by adding, if necessary, exceptional components of blow-ups of satellite points using the following criterion 
 and set, for each new $E_i$, $v_i=\min \{ v_{i} (a_1), v_{i} (a_2) \}$.
 
 {\textit{\textbf{Repeat}}}: 

\begin{enumerate}
 \item [$\cdot$] Blow-up the satellite point $q \in E_i \cap E_j$,
\end{enumerate}

{\textit{\textbf{while}}} $v_i(a_1)>v_i$ and $v_j(a_2)>v_j$.

\item Return $F$ and $\left(N_{f,i} = v_i(b) + v_i(a_1), N_{g,i}= v_i(b) + v_i(a_2)\right)_{i \in I}$.

\end{enumerate}


\end{algorithm}

\vskip 1mm

Observe that the strategy of performing further blow-ups from the log resolution of $f \cdot g$ 
until the irreducible components of the strict transform in $f$
and $g$ are separated by a dicritical component is rather vague. In contrast, Algorithm \ref{alg-log_resol} determines precisely the minimal blow-ups needed to achieve that.

\subsection{Computing the integral closure of an ideal}\label{ssec:unloading}

Let $\pi: Y \longrightarrow X$ be a proper birational morphism obtained after a sequence of point blowings-up.   An effective divisor with integral coefficients $D \in {\rm Div}(Y)$ is called \emph{antinef} if $-D\cdot E_i \geqslant 0$,  for every exceptional prime divisor $E_i$. Lipman \cite{Lip} established a one to one correspondence  between antinef divisors in ${\rm Div}(Y)$ and complete ideals in $\mathcal{O}_{X, O}$. Namely, given an effective divisor $D =  \sum d_i E_i \in \textrm{Div}_{\mathbb{Q}}(Y)$, we may consider its associated sheaf ideal $\pi_* \mathcal{O}_{X'}(-D)$ whose stalk at $O$ is
\begin{equation*} \label{eq:complete-ideal}
  H_D = \{h \in \mathcal{O}_{X, O}\ |\ v_i(h) \geqslant \lceil d_i \rceil\ \textrm{for all}\ E_i \leqslant D\} ,
\end{equation*}
where $v_i(h)$ is the value of  $h$ at $E_i$. These ideals are \emph{complete}, see \cite{Zar}, and $\mathfrak{m}$-primary whenever $D$ has exceptional support.

\vskip 1mm


The divisors 
$\lfloor \lambda_j  \tilde{F}_0 - K_{\pi} \rfloor$ used in the definition of multiplier ideals  are not antinef. Given a non-antinef divisor $D$, one can compute an antinef divisor defining the same ideal, called the antinef closure, via the so called unloading procedure (see \cite[\S 2.2]{AAD1} or  \cite[\S 4.6]{CAbook}). 
The unloading procedure incrementally expands the exceptional part $D_{\rm exc}$ of $D$ to obtain the minimal antinef divisor $D'$ containing $D$.
This is one of the key ingredients of the algorithm in \cite{AAD1}. The finiteness and correctness of the unloading procedure  it is proved there. 

\vskip 1mm

\begin{algorithm}[Unloading procedure]\hspace{45mm}

\vskip 1mm

\noindent {\tt Input:} A divisor $D$.

\vskip 1mm
\noindent {\tt Output:} The divisor $D'$ that is the antinef closure of $D$.

\vskip 2mm

{\textit{\textbf{Repeat}}}

\begin{enumerate}
 \item [$\cdot$] Define $\Theta:= \{E_i \leqslant D_{\rm exc} \hskip 2mm | \hskip 2mm  \rho_i= -\lceil D \rceil \cdot E_i <0 \}.$
 \item [$\cdot$] Let $n_i = \left \lceil \frac {\rho_i}{E_i^2} \right \rceil$ for $i\in\Theta$. Notice that $(\lceil D\rceil+n_i E_i)\cdot E_i\leqslant 0$.
 \item [$\cdot$] Define a new divisor as $D'= \lceil D \rceil + \sum_{E_i \in \Theta} n_i E_i$.
\end{enumerate}

{\textit{\textbf{Until}}} the resulting divisor $D'$ is antinef.

\end{algorithm}

\vskip 1mm

%

\subsection{Computing the generators of an ideal associated to an antinef divisor}\label{ssec:gen}

Once we have the antinef divisor that describes the corresponding multiplier ideal of the meromorphic function we may use the algorithm developed in \cite{AAB2} to compute a set of generators of the multiplier ideal. This step of our method does not require any adjustment for the case of meromorphic functions and we briefly recall it here.

\vskip 2mm

First, start with a divisor $D$, which we assume to be antinef. The divisor $D$ is decomposed into simple divisors $D_1, \dots, D_r$ by using Zariski's decomposition theorem \cite{Zar, CAbook}. This decomposition is precisely
\begin{equation*} \label{eq:divisor-decomp}
D = \sum \rho_i D_i, \quad \textrm{where} \quad \rho_i = -D\cdot E_i \geqslant 0,
\end{equation*}
and $H_{D_i}$ is a simple ideal appearing in the factorization of $H_D$ with multiplicity $\rho_i$. 

\vskip 2mm

Now, for each simple divisor $D_i$, compute $D'_i$ an antinef divisor defining an adjacent ideal $H_{D'_i} \subset H_{D_i}$, such that $D'_i$ is the antinef closure of $D_i + E_{O}$ ($E_{O}$ is the exceptional divisor obtained blowing-up the origin $O$). Next, find an element $f \in \mathcal{O}_{X, O}$ belonging to $H_{D_i}$ but not to $H_{D'_i}$. Now, $D'_i$ is no longer simple but has smaller support than $D_i$. This part is repeated with $D := D'_i$ until $H_D$ is the maximal ideal $\mathfrak{m} \subseteq \mathcal{O}_{X, O}$.

\vskip 1mm

This first part of the algorithm generates a tree where each vertex is an antinef divisor and where the leafs of the tree are all $\mathfrak{m}$. The second part traverses the tree bottom-up computing in each node the ideal associated to the divisor. Using the notations from the above paragraph, given any node in the tree with divisor $D$, the ideal $H_D$ is computed multiplying the ideals in child nodes $H_{D'_1} \cdots H_{D'_r}$ and adding the element $f$ to the resulting generators. 

\section{Skoda's theorem version for meromorphic functions} \label{Sec4}

Let $f \in \cO_{X,O}= \CC\{ x_1,\dots , x_n\}$ be an holomorphic function and consider the classical multiplier ideals $\cJ(f^\lambda)$. A version of Skoda's theorem for hypersurfaces  \cite[\S 9]{lazarsfeld2004positivity} states that 
$$
\cJ \left( {f}^{\lambda+\ell}\right)
=
{f^\ell} \cJ \left(f^\lambda\right)
$$
for every $\ell\in\NN.$ This implies a peridiocity of the set of jumping numbers, indeed   any jumping number is of the form $\lambda +\ell$ with $\lambda$ being a jumping number in the interval $[0,1]$ and $\ell\in\NN.$

\vskip 2mm

For the case of meromorphic functions we only have a weaker version of Skoda's theorem.

\begin{proposition}\cite[Proposition 6.5]{AGLN}\label{PropSkoda}
Let $f,g\in \cO_{X,O}$ be a nonzero elements and $\lambda\in \RR_{\geq 0}$.
Then,
$$
\cJ \left( \left( \frac{f}{g}\right)^{\lambda+\ell}\right)
=
\left( 
\frac{f^\ell}{g^\ell} \; \cJ \left( \left( \frac{f}{g}\right)^\lambda\right)\right)
\cap \cO_{X,O}
= 
f^\ell \left( \cJ \left( \left(\frac{f}{g}\right)^\lambda\right): g^{\ell} \right)
$$
for every $\ell\in\NN.$
In particular,
if  $\lambda+1$ is a jumping number, then $\lambda$ is a jumping number.
\end{proposition}

We also have the following property.

\begin{lemma} \cite[Lemma 7.5]{AGLN} \label{Lem: JN_integers} For any $\lambda \in \RR_{> 0}$ we have $\cJ(f^\lambda) \subseteq \cJ((\frac{f}{g})^\lambda)$. In addition, if $f$ and $g$ have no common factors, it holds $\cJ((\frac{f}{g})^n)=(f^n)$ for any  $n \in \mathbb{Z}_{>0}$.
\end{lemma}

Additionally,

\begin{lemma}\label{Lem: JN_all_integers}
    Let $f,g\in \CC\{ x, y\}$. If $f$ vanishes at the origin and has no common factor with $g$, then $\mathbb{Z}_{>0}$ are jumping numbers of the meromorphic function $f/g$.
\end{lemma}

\begin{proof}
    Let $E_i$ be an affine component of $f$. Then $N_i>0$, $k_i=0$, so $\lfloor \lambda_j N_i - k_i \rfloor = \lfloor\lambda_j \rfloor N_i = e_i^{\lambda_j}$, where the last equality is due to the unloading procedure not modifying the multiplicity of affine components. Following Algorithm \ref{Alg: JN meromorphic}, if $\lambda_j<n\in\mathbb{Z}_{>0}$ then $e_i^{\lambda_j} = \lfloor\lambda_j \rfloor N_i \neq n N_i = e_i^{n} \implies D_{\lambda_j}\neq D_n$. Therefore, $n$ is a jumping number.
\end{proof}

These results lead to some behaviours of the set of jumping numbers that contrast with the classical case. For the rest of this section we will only consider the case of dimension two, that is $\cO_{X,O}= \CC\{ x, y\}$, and we will use Algorithm \ref{Alg: JN meromorphic} to compute 
consecutive jumping numbers. In the sequel, we fix the following notation: 

\vskip 2mm
Let $\pi: Y \longrightarrow U$ be a log resolution of the meromorphic function  $f/g$ and consider the total transforms of $f$ and $g$.
\[ \tilde{f}=\pi^* f= \sum N_{f, i} E_i, \quad \quad \tilde{g}= \pi^* g= \sum N_{g, i} E_i. \]
Consider also the relative canonical divisor $K_\pi = \sum k_i E_i$.

\begin{theorem} \label{thm:main}
Let $f,g\in \CC\{ x, y\}$ be reduced holomorphic functions vanishing at the origin with no common factor. Then, there exist $n \in \mathbb{Z}_{>0}$ such that the only jumping numbers of the meromorphic function $f/g$ larger than $n$ are the integer numbers.
\end{theorem}

\begin{proof}
Since $f$ is reduced, for any affine component $E_i$ of $f$ we have $N_{f,i} = 1$. Moreover, since $f$ and $g$ do not have common factor, for these affine components we have $N_{g,i} = 0$ and thus $N_i=N_{f,i}-N_{g,i}=1$. The relative canonical divisor only has exceptional support so we also have $k_i=0$ on affine components.

\vskip 2mm

Now consider an integer jumping number $\lambda_{j-1}= n \in \mathbb{Z}_{>0}$ and recall that,
by Lemma \ref{Lem: JN_integers},
$$\cJ \left(\left(\frac{f}{g}\right)^n\right)=\left(f^n \right)= \pi_* \mathcal{O}_Y(- n \tilde{f}), $$
where $n \tilde{f} = \sum n N_{f, i}\, E_i$ is an antinef divisor. Then, following Algorithm \ref{Alg: JN meromorphic}, the next jumping number is




\[
    \lambda_j =  \min_{N_i>0}\left\{\frac{k_i+1+n N_{f, i}}{N_i} \right\}.
\]

We have that this next jumping number is $n+1$ if and only if for all $i$ such that $N_i>0$ we have: 
\begin{align*}
    n+1 \leqslant \frac{k_i+1+n N_{f, i}}{N_i}
    &\iff (n+1) (N_{f,i}- N_{g,i}) \leqslant k_i + 1 + n N_{f, i}
    \\&\iff N_{f,i} \leqslant (n+1) N_{g,i} + k_i + 1 .
\end{align*}
For the affine components of $f$ we have $N_{f,i}=1$ (and $N_{g,i}=0$, $k_i=0$) so the inequality is satisfied. We don't consider the affine components of $g$ since $N_i<0$. Since $g$ vanishes at the origin, for all exceptional components we have $N_{g,i}\geqslant 1$. Therefore, if $n$ is large enough, the inequality holds for all $i$ such that $N_i>0$, and thus $\lambda_j=\lambda_{j-1}+1= n+1$.

\vskip 2mm

Notice that, given a large enough integer jumping number $\lambda_{j-1}=n \in \mathbb{Z}_{>0}$ (which exists by Lemma \ref{Lem: JN_all_integers}), all the consecutive jumping numbers satisfy $\lambda_i = \lambda_{i-1} + 1 \in \mathbb{Z}_{>0}$ for all $i\geqslant j$.
\end{proof}

Using the notations we fixed in this section we also obtain the following:

\begin{corollary}
Let $f,g\in \CC\{ x, y\}$ be reduced holomorphic functions vanishing at the origin, with no common factor. Then, the jumping numbers of $f/g$ are precisely the set of integers $\mathbb{Z}_{>0}$ if and only if $N_{f,i} \leqslant N_{g,i} + k_i + 1$ for all $i$.
\end{corollary}

\section{Examples}  \label{Sec5}
In this section we provide examples where we compute the sequence of multiplier ideals of meromorphic functions using the algorithm developed in Section \ref{Sec2}. In particular, we will illustrate the phenomenon described in Section \ref{Sec4}. The algorithm has been implemented with the mathematical software  {\tt Magma} \cite{magma} and is available at: 
\begin{center}
  \url{https://github.com/rogolop}  
\end{center}

\begin{example}
    Consider the holomorphic functions:
    \begin{align*}
        &f = (y^2-x^3)^4 + x^8 y^5 ,
        \\ &g = y^2-x^3.
    \end{align*}
The values of $f$ and $g$ are collected in the following vectors:
{\Small
\begin{align*}
    &N_f = ( 8, 12, 24, 28, 31, 31, 60, 92, 124, 125, 31, 31, 31, 31, 31, 31, 31, 31, 31, 31, 31, 31, 31, 31, 31, 31, 31, 31, 31, 31, 31, 31 ),
    \\ &N_g = ( 2, 3, 6, 7, 8, 9, 15, 23, 31, 31, 10, 11, 12, 13, 14, 15, 16, 17, 18, 19, 20, 21, 22, 23, 24, 25, 26, 27, 28, 29, 30, 31 ).
\end{align*}
}
Therefore, we get the divisor $\tilde{F}_0$ given by the values
{\small $$N=( 6, 9, 18, 21, 23, 22, 45, 69, 93, 94, 21, 20, 19, 18, 17, 16, 15, 14, 13, 12, 11, 10, 9, 8, 7, 6, 5, 4, 3, 2, 1, 0 ).$$}
The jumping numbers of the meromorphic function 
    $f/g$ and the generators of the corresponding multiplier ideals are:
    \[
    \renewcommand{\arraystretch}{1.2}
    \setlength{\doublerulesep}{10\arrayrulewidth}
    \begin{array}{c c c}
    \begin{array}[t]{|l|l|}
        \hline 
        \bm{\lambda_j} & \bm{\cJ((f/g)^{\lambda_j})}
        \\[0mm]\hline\hline
                 5/18 & x, y
        \\\hline 35/93 &  x^2, y
        \\\hline 13/31 & x^2, x y, y^2
        \\\hline 43/93 & x^3, x y, y^2
        \\\hline 47/93 & g, x^2 y, x^3
        \\\hline 17/31 & g, x^4, x y^2, x^2 y
        \\\hline 55/93 & g, x^4, x^3 y
        \\\hline \bm{11/18} & x g, x^4, x^3 y, y g
        \\\hline 59/93 & x^2 y^2, x^5, x g, x^3 y, y g
        \\\hline 21/31 & x^5, x^4 y, x g, y g
        \\\hline \bm{22/31} & x^2 y^2, x^2 g, x^4 y, y g, x y^3
        \\\hline 67/93 & x^3 y^2, x^2 g, x^4 y, x^6, y g
        \\\hline \bm{70/93} & x^2 g, x^3 y^2, x^6, y^2 g, x y g, x^4 y
        \\\hline 71/93 & x^2 g, x^5 y, x^6, y^2 g, x y g
        \\\hline \bm{74/93} & x^3 y^2, x^2 y^3, x^3 g, y^2 g, x y g, x^5 y
        \\\hline 25/31 & x^7, x^3 g, y^2 g, x y g, x^5 y, x^4 y^2
        \\\hline \bm{26/31} & g^2, x^7, x^2 y g, x^3 g, x^5 y, x^4 y^2
        \\\hline 79/93 & g^2, x^7, x^2 y g, x^3 g, x^6 y
        \\\hline \bm{82/93} & g^2, x y^2 g, x^4 g, x^2 y g, x^3 y^3, x^6 y, x^4 y^2
        \\\hline 83/93 & g^2, x y^2 g, x^8, x^4 g, x^2 y g, x^5 y^2, x^6 y
        \\\hline \bm{86/93} & g^2, x^6 y, x^5 y^2, x^8, x^3 y g, x^4 g
        \\\hline 29/31 & g^2, x^7 y, x^3 y g, x^8, x^4 g
        \\\hline \bm{17/18} & x^7 y, x g^2, x^8, y g^2, x^3 y g, x^4 g
        \\\hline \bm{30/31} & x^7y, x^5y^2, xg^2, x^5g, yg^2, x^3y g, x^4y^3, x^2 y^2g
        \\\hline 91/93 & x^7 y, x^9, x g^2, x^5 g, y g^2, x^3 y g, x^6 y^2, x^2 y^2 g
        \\ \hline 1 & f
        \\\hline
    \end{array}
    &\quad
    \begin{array}[t]{|l|l|}
        \hline 
        \bm{\lambda_j} & \bm{\cJ((f/g)^{\lambda_j})}
        \\[0mm]\hline\hline
                 \bm{29/18} & f \cdot ( x, y )
        \\\hline \bm{53/31} & f \cdot ( y, x^2 )
        \\\hline \bm{163/93} & f \cdot ( x^2, x y, y^2 )
        \\\hline \bm{167/93} & f \cdot ( x^3, x y, y^2 )
        \\\hline \bm{57/31} & f \cdot ( g, x^2 y, x^3 )
        \\\hline \bm{175/93} & f \cdot ( g, x^4, x y^2, x^2 y )
        \\\hline \bm{179/93} & f \cdot ( g, x^4, x^3 y )
        \\\hline \bm{35/18} & f \cdot ( x g, x^4, x^3 y, y g )
        \\\hline \bm{61/31} & f \cdot ( x^2 y^2, x^5, x g, x^3 y, y g )
        \\\hline
                 2 & f^2
        \\\hline \bm{53/18} & f^2 \cdot ( x, y )
        \\\hline
                 3 & f^3
        \\\hline
                 4 & f^4
        \\\hline
                 5 & f^5
        \\\hline
                 6 & f^6
        \\\hline
                 7 & f^7
        \\\hline
                 8 & f^8
        \\\hline
                 9 & f^9
        \\\hline
        \multicolumn{2}{c}{\vdots}
    \end{array}
    \end{array}
    \]
\end{example}


Notice that the jumping numbers larger than $3$ are only integer numbers which illustrates Theorem \ref{thm:main}. We also point out that the version of Skoda's theorem \ref{PropSkoda} gives us
$$
\cJ \left( \left( \frac{f}{g}\right)^{\frac{29}{18}}\right)
=
\cJ \left( \left( \frac{f}{g}\right)^{\frac{11}{18}+ 1}\right)
= 
f \cdot \left( \cJ \left( \left(\frac{f}{g}\right)^{\frac{11}{18}}\right): g \right)
=
f \cdot \cJ \left( \left( \frac{f}{g}\right)^{\frac{5}{18}}\right).
$$
In this case $\cJ \left( \left( \frac{f}{g}\right)^{\frac{11}{18}}\right)$ is the first ideal that does not contain $g$. A similar phenomenon happens for the jumping numbers stated in bold at the table.

\vskip 3mm

\begin{example}
    Consider the holomorphic funcions:
    \begin{align*}
        &f = (y^2-x^3)^4 + x^8 y^5,
        \\&g' = y^2+x^3.
    \end{align*}
    The values of $f$ and $g$ are collected in the following vectors:
{\Small
\begin{align*}
    &N_f = ( 8, 12, 24, 28, 31, 60, 92, 124, 125, 24, 24, 24, 24, 24, 24, 24, 24, 24, 24, 24, 24, 24, 24, 24, 24, 24, 24 ),
    \\ &N_{g'} = ( 2, 3, 6, 6, 6, 12, 18, 24, 24, 7, 8, 9, 10, 11, 12, 13, 14, 15, 16, 17, 18, 19, 20, 21, 22, 23, 24 ).
\end{align*}
}
Therefore, we get the divisor $\tilde{F}_0$ given by the values
{\small $$N=( 6, 9, 18, 22, 25, 48, 74, 100, 101, 17, 16, 15, 14, 13, 12, 11, 10, 9, 8, 7, 6, 5, 4, 3, 2, 1, 0 ).$$}
    Comparing the jumping numbers of $f/g'$ with the previous example (in which $g = y^2-x^3$ and thus the log resolution of $f/g$ and $f/g'$ differ):
    
    \begin{table}[h!]
        \centering
        \renewcommand{\arraystretch}{1.5}
        \linespread{1.4}\selectfont
        \begin{tabular}{| p{0.5\linewidth} | p{0.5\linewidth} |}
        \hline
        $\bm{g' = y^2+x^3}$ & $\bm{g = y^2-x^3}$
        \\\hline\hline
        $\frac{27}{100}$, $\frac{7}{20}$, $\frac{39}{100}$, $\frac{43}{100}$, $\frac{47}{100}$, $\frac{51}{100}$, $\frac{11}{20}$, $\frac{29}{50}$, $\frac{59}{100}$, $\frac{63}{100}$, $\frac{
33}{50}$, $\frac{67}{100}$, $\frac{7}{10}$, $\frac{71}{100}$, $\frac{37}{50}$, $\frac{3}{4}$, $\frac{39}{50}$, $\frac{79}{100}$, $\frac{41}{50}$, $\frac{83}{100}$, $\frac{43}{50}$, $\frac{
87}{100}$, $\frac{89}{100}$, $\frac{9}{10}$, $\frac{91}{100}$, $\frac{47}{50}$, $\frac{19}{20}$, $\frac{97}{100}$, $\frac{49}{50}$, $\frac{99}{100}$
        &
        $\frac{5}{18}$, $\frac{35}{93}$, $\frac{13}{31}$, $\frac{43}{93}$, $\frac{47}{93}$, $\frac{17}{31}$, $\frac{55}{93}$, $\frac{11}{18}$, $\frac{59}{93}$, $\frac{21}{31}$, $\frac{22}{31}$, $\frac{
67}{93}$, $\frac{70}{93}$, $\frac{71}{93}$, $\frac{74}{93}$, $\frac{25}{31}$, $\frac{26}{31}$, $\frac{79}{93}$, $\frac{82}{93}$, $\frac{83}{93}$, $\frac{86}{93}$, $\frac{29}{31}$, $\frac{
17}{18}$, $\frac{30}{31}$, $\frac{91}{93}$
        \\\hline
        $1$, $\frac{151}{100}$, $\frac{159}{100}$, $\frac{163}{100}$, $\frac{167}{100}$, $\frac{171}{100}$, $\frac{7}{4}$, $\frac{179}{100}$, $\frac{91}{50}$, $\frac{183}{100}$, $\frac{
187}{100}$, $\frac{19}{10}$, $\frac{191}{100}$, $\frac{97}{50}$, $\frac{39}{20}$, $\frac{99}{50}$, $\frac{199}{100}$
        &
        $1$, $\frac{29}{18}$, $\frac{53}{31}$, $\frac{163}{93}$, $\frac{167}{93}$, $\frac{57}{31}$, $\frac{175}{93}$, $\frac{179}{93}$, $\frac{35}{18}$, $\frac{61}{31}$
        \\\hline
        $2$, $\frac{11}{4}$, $\frac{283}{100}$, $\frac{287}{100}$, $\frac{291}{100}$, $\frac{59}{20}$, $\frac{299}{100}$
        &
        $2$, $\frac{53}{18}$
        \\\hline
        $3$, $\frac{399}{100}$
        &
        $3$
        \\\hline
        $4$ & $4$
        \\\hline
        \multicolumn{1}{c}{\vdots} & \multicolumn{1}{c}{\vdots}
        \end{tabular}
    \end{table}
    
    The curves $g=y^2 - x^3$ and $g'= y^2+ x^3$ are analytically isomorphic. The difference in the jumping numbers is due to the different contact with the curve defined by $f$.
    The higher contact results in less jumping numbers.
\end{example}


\begin{example}
    Consider the holomorphic functions
    \begin{align*}
        &f = (y^2-x^3)^5 + x^{18},
        \\&g_k = (y^2-x^3)^k,
    \end{align*}
    with $k\in\mathbb{Z}_{\geqslant0}$. The contact of $g_k$ with $f$ increases with the exponent $k$, thus we get less jumping numbers.

    \begin{table}[htbp]
        \centering
        \renewcommand{\arraystretch}{1.5}
        \linespread{1.4}\selectfont
        \begin{tabular}{| p{0.3\linewidth} | p{0.25\linewidth} | p{0.12\linewidth} | p{0.07\linewidth} | p{0.07\linewidth} | p{0.07\linewidth} |}
        \hline
        $\bm{k=0}$ & $\bm{k=1}$ & $\bm{k=2}$ & $\bm{k=3}$ & $\bm{k=4}$ & $\bm{k\geqslant 5}$
        \\\hline\hline
        $\frac{1}{6}$, $\frac{41}{180}$, $\frac{23}{90}$, $\frac{17}{60}$, $\frac{14}{45}$, $\frac{61}{180}$, $\frac{11}{30}$, $\frac{71}{180}$, $\frac{19}{45}$, $\frac{77}{180}$, $\frac{9}{20}$, $\frac{41}{90}$, $\frac{43}{90}$, $\frac{29}{60}$, $\frac{91}{180}$, $\frac{23}{45}$, $\frac{8}{15}$, $\frac{97}{180}$, $\frac{101}{180}$, $\frac{17}{30}$, $\frac{53}{90}$, $\frac{107}{180}$, $\frac{37}{60}$, $\frac{28}{45}$, $\frac{113}{180}$, $\frac{29}{45}$, $\frac{13}{20}$, $\frac{59}{90}$, $\frac{121}{180}$, $\frac{61}{90}$, $\frac{41}{60}$, $\frac{7}{10}$, $\frac{127}{180}$, $\frac{32}{45}$, $\frac{131}{180}$, $\frac{11}{15}$, $\frac{133}{180}$, $\frac{34}{45}$, $\frac{137}{180}$, $\frac{23}{30}$, $\frac{47}{60}$, $\frac{71}{90}$, $\frac{143}{180}$, $\frac{73}{90}$, $\frac{49}{60}$, $\frac{37}{45}$, $\frac{149}{180}$, $\frac{151}{180}$, $\frac{38}{45}$, $\frac{17}{20}$, $\frac{77}{90}$, $\frac{13}{15}$, $\frac{157}{180}$, $\frac{79}{90}$, $\frac{53}{60}$, $\frac{161}{180}$, $\frac{9}{10}$, $\frac{163}{180}$, $\frac{41}{45}$, $\frac{83}{90}$, $\frac{167}{180}$, $\frac{14}{15}$, $\frac{169}{180}$, $\frac{19}{20}$, $\frac{43}{45}$, $\frac{173}{180}$, $\frac{29}{30}$, $\frac{44}{45}$, $\frac{59}{60}$, $\frac{89}{90}$, $\frac{179}{180}$
        &
        $\frac{5}{24}$, $\frac{41}{144}$, $\frac{23}{72}$, $\frac{17}{48}$, $\frac{7}{18}$, $\frac{61}{144}$, $\frac{11}{24}$, $\frac{71}{144}$, $\frac{19}{36}$, $\frac{77}{144}$, $\frac{9}{16}$, $\frac{41}{72}$, $\frac{43}{72}$, $\frac{29}{48}$, $\frac{91}{144}$, $\frac{23}{36}$, $\frac{2}{3}$, $\frac{97}{144}$, $\frac{101}{144}$, $\frac{17}{24}$, $\frac{53}{72}$, $\frac{107}{144}$, $\frac{37}{48}$, $\frac{7}{9}$, $\frac{113}{144}$, $\frac{29}{36}$, $\frac{13}{16}$, $\frac{59}{72}$, $\frac{121}{144}$, $\frac{61}{72}$, $\frac{41}{48}$, $\frac{7}{8}$, $\frac{127}{144}$, $\frac{8}{9}$, $\frac{131}{144}$, $\frac{11}{12}$, $\frac{133}{144}$, $\frac{17}{18}$, $\frac{137}{144}$, $\frac{23}{24}$, $\frac{47}{48}$, $\frac{71}{72}$, $\frac{143}{144}$
        & 
        $\frac{5}{18}$, $\frac{41}{108}$, $\frac{23}{54}$, $\frac{17}{36}$, $\frac{14}{27}$, $\frac{61}{108}$, $\frac{11}{18}$, $\frac{71}{108}$, $\frac{19}{27}$, $\frac{77}{108}$, $\frac{3}{4}$, $\frac{41}{54}$, $\frac{43}{54}$, $\frac{29}{36}$, $\frac{91}{108}$, $\frac{23}{27}$, $\frac{8}{9}$, $\frac{97}{108}$, $\frac{101}{108}$, $\frac{17}{18}$, $\frac{53}{54}$, $\frac{107}{108}$
        &
        $\frac{5}{12}$, $\frac{41}{72}$, $\frac{23}{36}$, $\frac{17}{24}$, $\frac{7}{9}$, $\frac{61}{72}$, $\frac{11}{12}$, $\frac{71}{72}$
        &
        $\frac{5}{6}$
        &
        \\\hline
        $1+($jumping numbers \newline between $0$ and $1)$
        &
        $1$, $\frac{35}{24}$, $\frac{221}{144}$, $\frac{113}{72}$, $\frac{77}{48}$, $\frac{59}{36}$, $\frac{241}{144}$, $\frac{41}{24}$, $\frac{251}{144}$, $\frac{16}{9}$, $\frac{257}{144}$, $\frac{29}{16}$, $\frac{131}{72}$, $\frac{133}{72}$, $\frac{89}{48}$, $\frac{271}{144}$, $\frac{17}{9}$, $\frac{23}{12}$, $\frac{277}{144}$, $\frac{281}{144}$, $\frac{47}{24}$, $\frac{143}{72}$, $\frac{287}{144}$
        & 
        $1$, $\frac{35}{18}$
        &
        $1$ & $1$ & $1$
        \\\hline
        $2+($jumping numbers \newline between $0$ and $1)$
        &
        $2$, $\frac{65}{24}$, $\frac{401}{144}$, $\frac{203}{72}$, $\frac{137}{48}$, $\frac{26}{9}$, $\frac{421}{144}$, $\frac{71}{24}$, $\frac{431}{144}$
        & 
        $2$ & $2$ & $2$ & $2$
        \\\hline
        $3+($jumping numbers \newline between $0$ and $1)$
        &
        $3$, $\frac{95}{24}$
        &
        $3$ & $3$ & $3$ & $3$
        \\\hline
        $4+($jumping numbers \newline between $0$ and $1)$
        &
        $4$ & $4$ & $4$ & $4$ & $4$
        \\\hline
        \multicolumn{1}{c}{\vdots} & \multicolumn{1}{c}{\vdots} & \multicolumn{1}{c}{\vdots} & \multicolumn{1}{c}{\vdots} & \multicolumn{1}{c}{\vdots} & \multicolumn{1}{c}{\vdots}
        \end{tabular}
    \end{table}
\end{example}

\newpage
\bibliographystyle{amsalpha} 
\def\cprime{$'$} \def\cprime{$'$} \def\cprime{$'$}
\providecommand{\bysame}{\leavevmode\hbox to3em{\hrulefill}\thinspace}
\providecommand{\MR}{\relax\ifhmode\unskip\space\fi MR }
\providecommand{\MRhref}[2]{%
  \href{http://www.ams.org/mathscinet-getitem?mr=#1}{#2}
}
\providecommand{\href}[2]{#2}


\end{document}